\documentclass{amsart}

\usepackage[utf8]{inputenc}

\usepackage{amsmath,amssymb,graphicx}

\newcommand{\A}{\mathbb{A}}
\newcommand{\CC}{\mathbb{C}}
\newcommand{\FF}{\mathbb{F}}

\newcommand{\NN}{\mathbb{N}}
\newcommand{\PP}{\mathbb{P}}
\newcommand{\QQ}{\mathbb{Q}}
\newcommand{\RR}{\mathbb{R}}
\newcommand{\ZZ}{\mathbb{Z}}

\newcommand{\cC}{\mathcal{C}}
\newcommand{\cR}{\mathcal{R}}

\newcommand{\tup}[1]{\overline{#1}}
\newcommand{\atup}{\tup{a}}
\newcommand{\xtup}{\tup{x}}

\newcommand{\valring}{\mathcal{O}_{\!K}}
\newcommand{\maxid}{\mathcal{M}_K}
\newcommand{\Qp}{\QQ_p}
\newcommand{\Zp}{\ZZ_p}
\newcommand{\muQp}{\mu_{\Qp}}
\newcommand{\muK}{\mu_{K}}

\newcommand{\res}{\operatorname{res}}
\newcommand{\ac}{\operatorname{ac}}
\newcommand{\cha}{\operatorname{char}}
\newcommand{\Th}{\operatorname{Th}}
\newcommand{\Pt}{\operatorname{Pt}}

\newcommand{\VF}{\operatorname{VF}}
\newcommand{\RF}{\operatorname{RF}}
\newcommand{\VG}{\operatorname{VG}}

\newcommand{\LDP}{L_\mathrm{DP}}

\newcommand{\Tloc}{T_\mathrm{loc}}

\newcommand{\cCu}{\cC_{\mathrm{u}}}
\newcommand{\muu}{\mu_{\mathrm{u}}}
\newcommand{\mumot}{\mu_{\mathrm{mot}}}

\newcommand{\cCmot}{\cC_{\mathrm{mot}}}

\newtheorem{thm}{Theorem}[section]
\newtheorem{almthm}[thm]{Almost-Theorem}
\newtheorem{cor}[thm]{Corollary}
\newtheorem{prop}[thm]{Proposition}

\theoremstyle{definition}
\newtheorem{defn}[thm]{Definition}

\newtheorem{ex}[thm]{Example}

\numberwithin{equation}{section}

\newenvironment{expl}{\par\smallskip\bgroup\footnotesize \noindent \emph{Explanation:}}{\par\smallskip\egroup}
\newenvironment{rem}{\par\smallskip\bgroup\footnotesize \noindent \emph{Remark:}}{\par\smallskip\egroup}

\title{An Introduction to Motivic Integration}
\author{Immi Halupczok}

\begin{document}

\maketitle

\tableofcontents

This introduction to motivic integration is aimed at readers who have some base knowledge of model theory of valued fields,
as provided e.g.\ by the notes by Martin Hils in this volume. I will not assume a lot of knowledge about valued fields.

\section{Introduction}

Given a non-archimedean local field like the field $\Qp$ of the $p$-adic numbers, one has a natural Lebesgue measure $\muQp$ on $\QQ_p^n$.
Motivic measure is an analogue of $\muQp$ which on the one hand also works in valued fields which do not have a classical Lebesgue measure and which on the other hand works in a field-independent way; motivic integration is integration with respect to that ``measure''.

The sets we want to measure are definable ones (in a suitable language of valued fields).
As an example, let $\phi_{1}(x)$ be the formula $v(x) \ge -1$, where $v$ is the valuation map.
An easy computation (assuming one knows how $\muQp$ is defined) shows that the measure of the set defined by $\phi_{1}$ is
$\muQp(\phi_{1}(\Qp)) = p$ for every $p$.
The same computation also works in any other non-archimedean local field $K$, yielding that $\muK(\phi_{1}(K))$ is equal to the
cardinality of the residue field of $K$. (I will define in Subsection~\ref{sect.val.fld} what a non-archimedean local field is.)

For other formulas $\phi(x)$, the measure of $\phi(K)$ might depend on $K$ in a more complicated way, but it
turns out that $\mu_K(\phi(K))$ can always be expressed in terms of cardinalities of some definable subsets of the residue field.
(This is true only under some assumptions about $K$; in this introduction, I will just write ``for suitable $K$''.)
In this sense, the measure of $\phi(K)$ can be described uniformly in $K$:

\smallskip
{
\leftskip 3em\noindent
Motivic measure expresses the \emph{measure} of a set \emph{in the valued field}\\
in terms of \emph{cardinalities} of sets \emph{in the residue field}.
\par
}
\smallskip

More formally, the motivic measure $\mumot(\phi)$
of a formula $\phi$
is an element of a variant $\cCmot^0$ of the Grothendieck ring of formulas in the ring language, where
the class $[\psi] \in \cCmot^0$ of a ring formula $\psi$ stands for the cardinality of the set defined by $\psi$
in the residue field. For example, for our above example formula $\phi_{1}$, we have $\mumot(\phi_{1}) = [\psi_{1}]$,
where $\psi_{1}$ is a formula defining the residue field itself.

Once measures are expressed uniformly in this way, one can also make sense of this
in valued fields which do not have a Lebesgue measure.
For instance, consider the field $K:= \CC((t))$ of formal power series with complex coefficients
and consider again our formula $\phi_{1}$ from above. As in $\Qp$,
a Lebesgue measure of $\phi_{1}(K)$ would have to be equal to the number of elements of the residue
field, which is $\CC$ in this case. Since $\CC$ is infinite, one can deduce that no (non-trivial translation-invariant)
Lebesgue-measure exists on $K$. However, one can make sense of $\muK(\phi_{1}(K))$ as an element of the Grothendieck ring
of definable sets in $\CC$, namely $\muK(\phi_{1}(K)) = [\psi_{1}(\CC)] = [\CC]$.
Getting such a kind of measure on $\CC((t))$ was the original goal of motivic integration, as invented by Kontsevich.
Indeed, that measure allowed Kontsevich \cite{Kon.mot} to give a simpler an more conceptual proof of a result by Batyrev about invariants of certain manifolds.

\medskip

Once one has a measure, one would also like to be able to integrate.
Lebesgue integration allows us to integrate functions from $K^n$ to $\RR$
for non-archimedean local fields $K$. Motivic integration
should allow us to do this uniformly in $K$ and moreover to generalize
this to other $K$. To this end, we need a field-independent way of
specifying functions $K^n \to \RR$.
This is done by introducing abstract rings of
``motivic functions''; such a motivic function $f$ determines
actual function $f_K\colon K^n \to \RR$ for every suitable $K$,
and the ``motivic integral'' of such an $f$ is an element of the same ring $\cCmot^0$ as above,
expressing the values of the integrals $\int_{K^n} f_K\,d\muK$ for all suitable $K$
in terms of cardinalities of sets in the residue field.

Again, this first allows us to uniformly integrate in all (suitable) non-archimedean local fields
and then also yields a notion of integration in other fields $K$ like $\CC((t))$.
However, for such $K$, the objects we are integrating are not functions $K^n \to \RR$ anymore.
Since the measure on $\CC((t))$ takes values in $\cCmot^0$, one would expect that also the functions
should take values in $\cCmot^0$. This is a good approximation to the truth, but in reality,
to obtain a smoothly working formalism, one needs to work with more abstract
objects than mere functions. The reward is that in many ways, motivic integration behaves like normal integration:
it satisfies a version of the Fubini Theorem and a change of variables formula.

\medskip

In these notes, 
after fixing notations and conventions in Section~\ref{sect.notn}, I will spend three sections
on ``uniform $p$-adic integration''. This is a weak version of motivic integration, which
provides a field independent way of integrating in non-archimedean local fields, but which does not generalize
to other valued fields. There is a whole range of applications for which
uniform $p$-adic integration is already strong enough, and I will give
one such application as a motivation, namely counting congruence classes of solutions of polynomial equations.
The benefit of restricting to uniform $p$-adic integration in these notes is that it can be defined
in a much more down-to-earth way than ``full'' motivic integration, while
many key aspects can already be seen on this version.
In the last section, I will sketch how to get from uniform $p$-adic integration
to motivic integration.

\section{Notation and language}
\label{sect.notn}

\subsection{The valued fields}
\label{sect.val.fld}


Throughout these notes, we will use the following notation:

\begin{itemize}
 \item $K$ is a henselian valued field with value group $\ZZ$ (``Henselian'' means that the conclusion of Hensel's Lemma holds;
 see below for examples.)
 \item $\valring \subseteq K$ is its valuation ring.
 \item $\maxid \subseteq \valring$ is the maximal ideal.
 \item $v\colon K \to \ZZ \cup \{\infty\}$ is the valuation map.
 \item $k$ is the residue field of $K$.
 \item $\res\colon \valring \to k$ is the residue map.
 \item $p \in \PP \cup \{0\}$ always stands for the residue characteristic of $K$, i.e., the characteristic of $k$ (here, $\PP$ denotes the set of primes).
 \item $q$ is the cardinality of $k$. (Usually, $k$ will be finite, and hence $q = p^r$ for some $r$).
 \item $\ac\colon K \to k$ is an angular component map. Formally, this means that $\ac$ is a group homomorphism from $K^\times$ to $k^\times$
      which agrees with $\res$ on $\valring^\times$, extended by $\ac(0) := 0$. The fields $K$ we will be interested in
      have natural angular component maps (associating to a series the most significant coefficient); see below.
\end{itemize}

To various of the above objects, we might sometimes add an index $K$ to emphasize the dependence on $K$,
writing e.g.\ $k_K$ for the residue field and $q_K$ for the cardinality of $k_K$.

If $K$ has characteristic $0$, the residue characteristic can either also be $0$, in which case we say that
$K$ has ``equi-characteristic $0$'', or it is $p \in \PP$; in that case, we say that $K$ has ``mixed characteristic''.
(If $K$ has characteristic $p \in \PP$, then $k$ also has characteristic $p$.)

The main examples of valued fields we are interested in are the following; all of them are complete and hence henselian
(by Hensel's Lemma):

\begin{ex}\label{ex.qp}
The $p$-adic numbers
\begin{equation}\label{eq.Qp}
K = \QQ_p = \{\underbrace{\sum_{i=N}^{\infty} a_i p^i}_{=a} \mid N \in \ZZ, a_i \in \{0, \dots, p-1\}\}. 
\end{equation}
Here, the residue field $k$ is $\FF_p$, and assuming
$a_N \ne 0$ in (\ref{eq.Qp}), we have $v(a) = N$ and $\ac(a) = a_N$. The field $\QQ_p$ has mixed characteristic.
\end{ex}

\begin{ex}\label{ex.pow}
The field
\begin{equation}\label{eq.kt}
K = k((t)) = \{\underbrace{\sum_{i=N}^{\infty} a_i t^i}_{=a} \mid N \in \ZZ, a_i \in k\}
\end{equation}
of formal power series over any field $k$.
As the notation suggests, $k$ is the residue field, and again, assuming
$a_N \ne 0$, we have $v(a) = N$ and $\ac(a) = a_N$. This $K$ either has positive characteristic (if $\cha k > 0$)
or equi-characteristic $0$ (if $\cha k = 0$).
\end{ex}

Valued fields which are locally compact (in the valuation topology) will play a particular role for us,
since on these, one has a Lebesgue measure. Such fields are called non-archimedean local fields. In the following,
I will just write ``local field'' (omitting ``non-archimedean''), since we are not interested in the archimedean ones.

\begin{prop}
Exactly the following valued fields are local fields:
\begin{itemize}
 \item the $p$-adic numbers $\Qp$;
 \item the power series fields $\FF_p((t))$ (where $\FF_p$ is the finite field with $p$ elements);
 \item finite extensions of any of the above.
\end{itemize}
\end{prop}

\subsection{The language}
\label{sect.lang}

We will consider $K$ as a structure in a suitable language. Since we are not interested in syntactic properties,
the precise language does not matter, provided that it yields the right definable sets. Let us fix a convenient
language nevertheless, namely the Denef--Pas language $\LDP$, which is a three-sorted language consisting of the following:

\begin{itemize}
 \item one sort $\VF$ for the valued field itself, with the ring language $\{+, -, \cdot, 0, 1\}$ on it;
 \item one sort $\RF$ for the residue field, also with the ring language;
 \item one sort $\VG$ for the value group, with the language $\{+, -, 0, <\}$ of ordered abelian groups;
 \item the valuation map $v \colon \VF \to \VG \cup \{\infty\}$;
 \item the angular component map $\ac\colon \VF \to \RF$.
\end{itemize}

We will use the notations $\VF$, $\RF$, $\VG$ (instead of $K$, $k$, $\ZZ$) if we want to speak about the sorts without
fixing a specific valued field.

\subsection{Definable sets}
\label{sect.def}

By a ``definable set'' $X$, we will mostly mean a field-independent object like $\VF$, $\RF$, $\VG$:
Such an $X$ is in reality just a formula, but using different notation:
$X_K$ is the set defined by the formula in a structure $K$, and we use set theoretic notation
like $X \cap Y$, $X \times Y$, etc.\ for definable sets $X, Y$.
In a similar way, given two definable sets $X, Y$,
by a ``definable function $f\colon X \to Y$'', we mean a formula
defining a function $f_K\colon X_K \to Y_K$ for every $K$.

We will always work in some fixed theory $T$ (see the next subsection). Whenever we write statements like
$X = Y$ or $X \subset Y$ for definable sets $X$, $Y$, we mean that $X_K = Y_K$ or $X_K \subset Y_K$ holds for every $K \models T$.
(In particular, ``$X$'' is really a formula up to equivalence modulo $T$.)




\subsection{The theory}
\label{sect.th}

In Sections~\ref{sect.meas}, \ref{sect.int} and \ref{sect.proof}, the fields we will be interested in will be
``local fields of sufficiently big residue characteristic''. We denote the ``corresponding'' theory by $\Tloc$:
\[
\Tloc := \bigcup_{K\text{ local field}}\Th(K) \cup \{\cha k \ne p \mid p \in \PP\}.
\]
Indeed, a sentence follows from $\Tloc$ if and only if it holds in all local fields $K$
of sufficiently big residue characteristic. (Here, the implication ``$\Rightarrow$'' uses compactness.)
In particular, according to Subsection~\ref{sect.def}, for definable $X, Y$, ``$X = Y$''
means that we have $X_K = Y_K$ for all local fields $K$ of sufficiently big residue characteristic.

The theory $\Tloc$ can also be described more explicitly: it is the theory of henselian valued fields with value group elementarily equivalent to $\ZZ$
and with pseudo-finite residue field of characteristic $0$.
In Section~\ref{sect.mot}, we will also consider a theory $T_0 \subset \Tloc$, which is
the same as $\Tloc$ except that the condition that the residue field is pseudo-finite has been dropped.
In particular, for any $k$ of characteristic $0$, $K := k((t))$ is a model of $T_0$.

\subsection{A key proof ingredient: quantifier elimination}
\label{sect.qe}

The proofs in these notes use various ingredients, but all those ingredients follow
from one single result, namely Denef--Pas Quantifier Elimination. Even though I will
not really explain how quantifier elimination implies the ingredients, I feel that I should
at least state it:


\begin{thm}[{\cite[Theorem~4.1]{Pas.cell}}]\label{thm.qe}
Any $\LDP$-formula is equivalent, modulo $T_0$, to an $\LDP$-formula without quantifiers running over $\VF$.
\end{thm}

\begin{rem}
The formulation of \cite[Theorem~4.1]{Pas.cell} sounds as if this kind of quantifier elimination is only
obtained in each model $K$ of $T_0$ individually. However, all proofs are uniform in $K$, as stated at
the beginning of \cite[Section~3]{Pas.cell}. Also, many other accounts of quantifier elimination directly state
the stronger version.
\end{rem}

Using that the only symbols in $\LDP$ connecting the different sorts are the valuation map and
the angular component map, one obtains a rather precise description of formulas without 
$\VF$-quantifiers and hence also of sets defined by such formulas.

\section{Measuring}
\label{sect.meas}

As already stated, in this section (and also in the next two), we are interested in
``local fields of sufficiently big residue characteristic'' and hence we work in the theory $\Tloc$ (see Subsection~\ref{sect.th}).
In particular, $K$ will always be a local field (and we will always use the notation from Subsection~\ref{sect.val.fld}).

\subsection{Motivation: Poincaré series}

Let me start by introducing a question which will serve as a motivating application.

Let $V$ be an affine variety defined over $\ZZ$, say, given by polynomials
$f_1, \dots, f_\ell \in \ZZ[\xtup]$ in variables $\xtup = (x_1, \dots, x_n)$.
We use the usual notation from algebraic geometry for ``$R$-rational points of $V$'':
For any ring $R$ (commutative, with unit), we write
\[
V(R) = \{\xtup \in R^n \mid f_1(\xtup) = \dots = f_\ell(\xtup) = 0\}.
\]
A problem coming from number theory consists in determining the cardinalities $N_m := \#V(\ZZ/m\ZZ)$ for $m \in \NN$.
Using the Chinese Remainder Theorem, this can be reduced to the case where $m = p^s$ for $p \in \PP$ and $s \in \NN$.
Now one question is: How does $N_{p^s}$ depend on $p$ and on $s$?

To understand the dependence on $s$ for a fixed prime $p$, one considers the associated Poincaré series:
\begin{defn}\label{defn.poin}
The \emph{Poincaré series} associated to $V$ and $p$ is the formal power series
\[
P_{V,p}(T) := \sum_{s = 0}^\infty N_{p^s} T^s \in \QQ[[T]].  
\]
\end{defn}

An intriguing result by Igusa (later generalized by Denef and Meuser) is that this series is a rational function in $T$:
\begin{thm}\label{thm.poin}
$P_{V,p}(T) = g_p(T)/h_p(T)$ for polynomials $g_p(T), h_p(T) \in \QQ[T]$.
\end{thm}
\begin{expl}
The equation makes sense in the field $\QQ((T))$. Another way of stating it is: If one formally multiplies the
power series $P_{V,p}(T)$ by the polynomial $h_p(T)$, the series one obtains is actually a polynomial, namely $g_p(T)$.
\end{expl}

\begin{ex}
If $V = \A^1$ (i.e., $\xtup = x_1$ and no polynomial equation at all), we have $N_{p^s} = p^{s}$ and hence
\[
P_{V,p}(T) = \sum_{s\ge 0} p^sT^s = \frac{1}{1-pT}.
\]
\end{ex}

The two polynomials $g_p(T)$ and $h_p(T)$ together entirely determine how $N_{p^s}$ depends on $s$,
so the next question is how $g_p(T)$ and $h_p(T)$ depend on $p$.
The first claim is that the degrees of $g_p$ and $h_p$ can be bounded independently
of $p$. Moreover, one can describe how their coefficients depend on $p$: the ones of $h_p$ are just polynomials
in $p$, and those of $g_p$ are given by cardinalities of definable sets in the residue field.
To avoid some technicalities, we make these claims only for sufficiently big $p$.
Here is the precise statement:

\begin{thm}\label{thm.poin.u}
Let $V$ be an affine variety defined over $\ZZ$ (as before).
Then there exist ring formulas $\phi_0, \dots, \phi_d \subset \RF^r$ 
and a polynomial $h \in \QQ[p,T]$ such that
for $p \gg 1$, we have
\[
P_{V,p}(T) = \frac{\sum_{i=0}^d \#\phi_{i}(\FF_p)\cdot T^i}{h(p,T)}.
\]
\end{thm}

In these notes, we will show how Theorem~\ref{thm.poin.u} can be proven using uniform $p$-adic integration, which
we will start introducing now.

\begin{rem}
Readers familiar with Poincaré series will note that only rather specific polynomials can arise as $h(p,T)$.
One does obtain this using the methods presented in these notes; I am omitting this only for simplicity
of the presentation.
\end{rem}

\subsection{Uniform $p$-adic measure}

Let us first fix a local field $K$, e.g.\ $K = \Qp$.
On such a $K$, there is a unique translation invariant measure $\muK$ that associates the measure $1$ to the valuation ring $\valring$.

\begin{figure}
\includegraphics{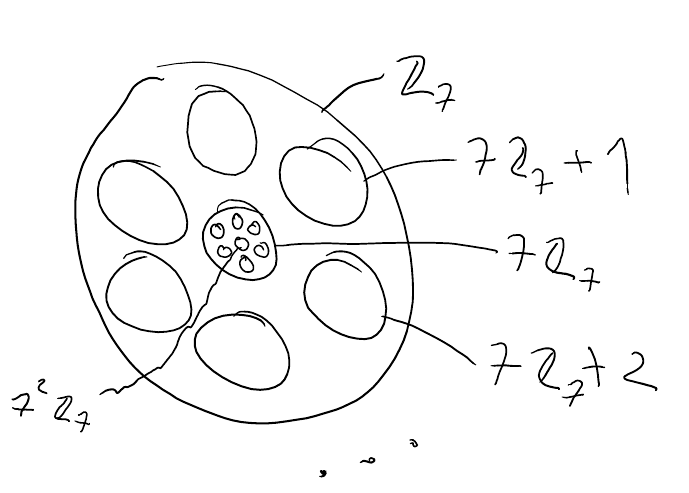}
\caption{$\ZZ_7$ is the union of $7$ translates of $7\ZZ_7$; $7\ZZ_7$ is the union of $7$ translates of $7^2\ZZ_7$; etc.}
\label{fig.measure}
\end{figure}

\begin{expl}
Existence and uniqueness of $\muK$ follows very generally from the fact that $(K,+)$ is a locally compact topological group
($\muK$ is the Haar measure of that group), but it can also easily be seen in a down-to-earth way.
In the case $K = \QQ_p$, for example, we define $\muQp(\Zp) := 1$. Then, using that $\Zp$ is the disjoint union of $p$ translates
of $p\Zp$, we deduce $\muQp(p\Zp) = p^{-1}$, and then, in a similar way, $\muQp(p^r\Zp) = p^{-r}$ for any $r \in \ZZ$ (see Figure~\ref{fig.measure}).
Arbitrary measurable sets can then be approximated by disjoint unions of such balls.
\end{expl}

\begin{ex}\label{ex.qu}
If $p \ge 3$, then 
the measure of the set $X := \{x^2 \mid x \in \Zp\}$ of squares in the valuation ring is
$\frac{p}{2(p+1)}$. This can be obtained as follows. First one proves, using Hensel's Lemma, that
an element $x\in \Zp$ is a square if and only if $v(x)$ is even and $\ac(x)$ is a square in the residue field $k = \FF_p$ (see Figure~\ref{fig.squares}).
Thus $X$ is the disjoint union of the sets $X_{r,a} := \{x \mid v(x) = r, \ac(x) = a\}$, where $r$ runs over $2\NN$ and $a$ runs
over the non-zero squares in $\FF_p$. (More precisely, $X$ additionally contains $0$, but $\muQp(\{0\}) = 0$.)
Now $\muQp(X_{r,a}) = p^{-r-1}$ and $\FF_p$ contains $\frac{p-1}2$ non-zero squares, so
\[
\muQp(X) = \frac{p-1}2\sum_{i = 0}^\infty p^{-2i-1},
\]
which, as a little computation shows, is equal to $\frac{p}{2(p+1)}$.
\end{ex}

\begin{figure}
\includegraphics{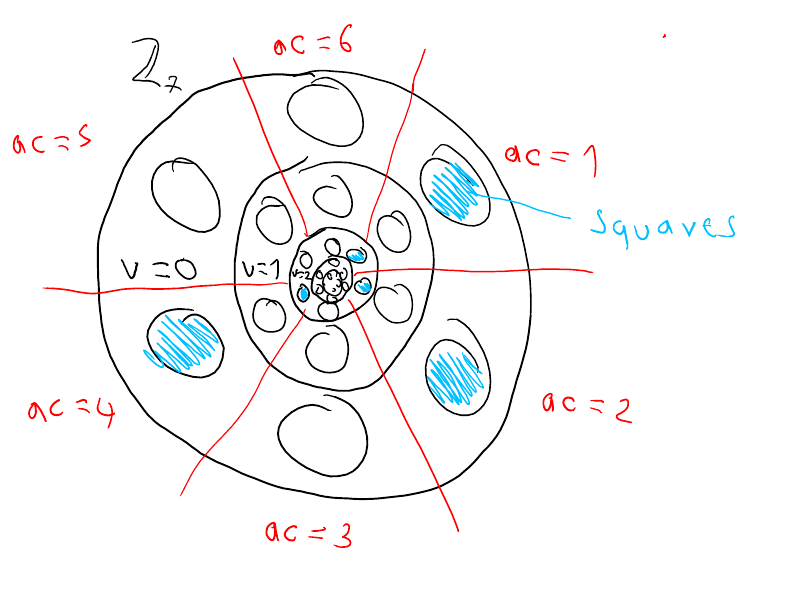}
\caption{Squares in $\ZZ_7$: The valuation has to be even, and the angular component has to be a square in $\FF_7$ (i.e., 1, 2 or 4).}
\label{fig.squares}
\end{figure}

We want to measure definable sets. It is not clear whether definable sets are always measurable
in local fields $K$ of positive characteristic (since the model theory of those fields
is not understood), but from quantifier elimination (Theorem~\ref{thm.qe}), one can deduce the following:

\begin{prop}
Given any definable set $X \subset \VF^n$, the set $X_K$ is measurable for any local field $K$ with $p \gg 1$
(i.e., of sufficiently big residue characteristic, where the bound might depend on $X$).
\end{prop}

\begin{rem}
The bound on $p$ is not needed in mixed characteristic, but in these notes, we will only be interested in big $p$ anyway.
\end{rem}

Now we would like to know: Given a definable set $X \subset \VF^n$,
how does the measure $\muK(X_K)$ depend on $K$ for $p \gg 1$?

To make this question more formal,
we consider the ring $\cR^0$ consisting of tuples $(a_K)_K$, $a_K \in \RR$, where $K$ runs
over all non-archimedean local fields, and where two tuples are identified if they agree for all $K$ of sufficiently
big residue characteristic.

We define the ``uniform measure'' of a definable set $X$ to be
\[
\muu(X) := (\mu_K(X_K))_K \in \cR^0,
\]
provided that $\mu_K(X_K) < \infty$ for all $p \gg 1$.
(Note that this is a well-defined element of $\cR^0$ even though $X_K$ might not be measurable for small $p$; moreover,
two definable sets which we identify according to Subsection~\ref{sect.def} have the same uniform measure.)

Our goal is now to prove that for every $X$, $\muu(X)$ already lies in a subring $\cCu^0 \subset \cR^0$ which
is much smaller than $\cR^0$ and given very explicitly:

\begin{defn}\label{defn.cCu0}
Let $\cR^0$ be as defined above, and
let $\cCu^0 \subset \cR^0$ be the subring
%
generated by the following tuples:
\begin{enumerate}
 \item 	
   $(\#Z_K)_K$, where $Z \subset \RF^\ell$ is a definable set (for any $\ell$); and
 \item 
   $(1/h(q_K))_K$, where $h \in \ZZ[q]$ is a polynomial and $q_K$ is the
   cardinality of the residue field of $K$.
\end{enumerate}
\end{defn}


As announced, our aim is to prove:

\begin{thm}\label{thm.measu}
Suppose that $X \subset \VF^n$ is a definable set such
that $\muK(X_K) < \infty$ for all $K$ with $p \gg 1$. Then $\muu(X) \in \cCu^0$.
\end{thm}

Here are two examples motivating the generators of $\cCu^0$:

\begin{ex}
Let $Z$ be any definable subset of $\RF^n$, and let $X := \res^{-1}(Z)$ be its preimage in $\valring^n$.
Then an easy computation shows that for any $K$, we have
$\mu_K(X_K) = q_K^{-n}\#Z_K$. Thus $\muu(X)$ is equal to the product of $(\#Z_K)_K$ (a generator of the form (1)) and $(1/q_K^n)_K$ (a generator of the form (2)).
\end{ex}

This example shows that all elements of the form (1) are needed in $\cCu^0$. Note that the numbers $\#Z_K$ may depend on $K$
in a quite complicated way; even if $Z$ is a variety over the residue field, it is not really understood how $\#Z_K$
depends on the finite field $k$. In the entire theory developed in these notes, the functions $K \mapsto \#Z_K$ are used as a black box.

The following example shows that one also needs more complicated polynomials in (2):

\begin{ex}
Let $X$ be the set of squares in the valuation ring. The same computation as in Example~\ref{ex.qu}
shows that whenever the residue characteristic $p$ is at least $3$, we have
$\mu_K(X_K) = \frac{q_K}{2(q_K+1)}$.
Thus $\muu(X)$ is equal to the product of $(\#\RF_K)_K$ (a generator of the form (1)) and $(1/(2(q_K+1)))_K$ (a generator of the form (2)).
\end{ex}

Apart from asking about the measure of a single definable set $X$, we can also ask how the measure varies in a definable family,
i.e., given a definable set $X \subset S \times \VF^n$ (where $S$ is a definable set living in any sorts), 
how does
the measure $\muQp(X_{\Qp, s})$ of the fiber $X_{\Qp, s} = \{x \in \Qp^n \mid (s, x) \in X_{\Qp}\}$ depend on $s \in S_{\Qp}$?
This will be needed for our application to Poincaré series.

Before getting back to Poincaré series, let me mention
a nice consequence of Theorem~\ref{thm.measu}, namely an Ax--Kochen/Ershov transfer principle for measuring:
\begin{cor}
Given any definable set $X \subset \VF$, there exists an $N$ such that
if $K$ and $K'$ are local fields with the same residue field $k$ and $k$ has characteristic $\ge N$,
then $\mu_K(X_K) = \mu_{K'}(X_{K'})$.
\end{cor}

To see this, it suffices to note that for $a = (a_K)_K \in \cCu^0$, $a_K$ only depends on $k$ (for $p \gg 1$).
For the generators (2) in Definition~\ref{defn.cCu0}, this is immediately clear; for the generators (1),
this follows from the classical Ax--Kochen/Ershov transfer principle (or from quantifier elimination).

Results further below in these notes imply various other Ax--Kochen/Ershov like results, but I will
not go further into this.

\subsection{Application to Poincaré series}

Recall that we want to understand how $N_{p^s} := \#V(\ZZ/p^s\ZZ)$ depends on $p$ and $s$, where $V$ is an affine variety
given by polynomials $f_1, \dots, f_\ell$ (in $n$ variables). We will now express these cardinalities as measures of definable sets.
For this, first
note that we have an isomorphism of rings $\ZZ/p^s\ZZ \cong \Zp/p^s\Zp$. Then we have
\[
V(\ZZ/p^s\ZZ) = V(\Zp/p^s\Zp) = \underbrace{\{ \xtup \in \Zp^n \mid f_1(\xtup), \dots, f_\ell(\xtup) \in p^s\Zp\}}_{=: X_{\Qp,s}}/\mathord{\sim},
\]
where $\atup \sim \atup'$ iff $\atup - \atup' \in (p^s\Zp)^n$, and where $X_{\Qp,s}$ is a union of entire $\sim$-equivalence classes.
Each such equivalence class has measure $\muQp((p^s\Zp)^n) = p^{-sn}$, so
\begin{equation}\label{eq.mu.poin}
N_{p^s} = \#V(\Zp/p^s\Zp) = p^{sn}\muQp(X_{\Qp,s}).
\end{equation}
Note also that $X_{\Qp,s}$ is a definable family of sets, parametrized by $s$ as an element of the value group.
Now we can formulate a result similar to Theorem~\ref{thm.poin.u} for arbitrary such families:

\begin{thm}\label{thm.murat.u}
Suppose that $X \subset \VG_{\ge 0} \times \VF^n$ is a definable set,
and suppose that for every local field $K$ with $p \gg 1$ and
for every $s \in \NN$, we have 
$\muK(X_{K,s}) < \infty$. Then there exist definable sets $Z_0, \dots, Z_d \subset \RF^r$ 
and a polynomial $h \in \QQ[q,T]$ such that
whenever the residue characteristic of $K$ is sufficiently big, we have
\begin{equation}\label{eq.murat}
\sum_{s = 0}^\infty \muK(X_{K,s}) T^s = \frac{\sum_{i=0}^d \#Z_{i,K}T^i}{h(q_K,T)},
\end{equation}
where $q_K$ is the cardinality of the residue field.
\end{thm}

Theorem~\ref{thm.murat.u} implies Theorem~\ref{thm.poin.u}:
\begin{itemize}
 \item Even though the Poincaré series in Definition~\ref{defn.poin} is not exactly equal to the left hand side of (\ref{eq.murat}),
       due to the factor $p^{sn}$ in (\ref{eq.mu.poin}), it is obtained from (\ref{eq.murat})
       by a substitution $T \mapsto p^nT$; such a substitution does not change the nature of the right hand side of (\ref{eq.murat}).
       (Note that we work with $K =\Qp$ and hence $q_K = p$.)
 \item On the right hand side of (\ref{eq.murat}), we use cardinalities of sets in the residue field which are $\LDP$-definable;
       the claim of Theorem~\ref{thm.poin.u} is that one can take sets definable in the ring language.
       It can be deduced from quantifier elimination (Theorem~\ref{thm.qe}) that this does not make a difference, i.e.,
       that any $\LDP$-definable set in the residue field is 
       definable in the pure ring language (for $p \gg 1$).
\end{itemize}

Thus now our goal is to prove Theorem~\ref{thm.murat.u}.

\section{Integrating}
\label{sect.int}

\subsection{Uniform $p$-adic integration}
\label{sect.intu}

To understand the measure of a definable set, we will integrate out one variable after the other.
For example, the measure of a set $Z_{\Qp} \subset \Qp^2$ will be determined by first
measuring the fibers $Z_{\Qp,x} = \{y \in \Qp \mid (x,y) \in Z_{\Qp}\}$
and then integrating:
\[
\muQp(Z_{\Qp}) = \int_{\Qp} \muQp(Z_{\Qp,x})\,dx.
\]
This approach has the advantage that we can treat
one dimension at a time; however, it means that instead of just measuring, we also need
to be able to integrate uniformly in $K$.

To make sense of such uniform integration, we need a
way to uniformly specify functions $K^n \to \RR$.
We do this in a way similar as we defined $\cCu^0$:
Given a definable set $X$, we let $\cR(X)$ be the ring of tuples
$(f_K)_K$, where $K$ runs over all local fields and $f_K$ is a function from $X_K$ to $\RR$,
and where two tuples are identified if they agree for big $p$.
We will define a sub-ring $\cCu(X) \subset \cR(X)$ the elements of which we call ``motivic functions'',
and we will prove that those motivic functions can be integrated uniformly in a similar way as
we already measured definable sets uniformly. More precisely, those rings are closed under
partial integration:

\begin{thm}\label{thm.intu}
Suppose that $S$ and $X \subset S \times \VF^n$ are definable sets, that
$f \in \cCu(X)$ is a motivic function (as we will define below) and that for every $K$ with $p \gg 1$ and for every $s \in S_K$,
the function $x \mapsto f_K(s, x)$ is $L^1$-integrable on the fiber $X_{K,s}$. Then the tuple $g = (g_K)_K$
of functions $g_K\colon S_K \to \RR$ given by
\[
g_K(s) =\int_{X_{K,s}} f_K(s, x)\,dx
\]
is an element of $\cCu(S)$.
\end{thm}

\begin{expl}
By ``$L^1$-integrable'', I just mean that the integrals are finite and that they are not some kind of improper integrals;
that the functions are measurable will follow anyway from the definition of $\cCu(Z)$.
\end{expl}


And here is the definition of the rings $\cCu(X)$:

\begin{defn}\label{defn.cCun}
Fix a definable set $X$ (in any sorts) and let $\cR(X)$ be as above.
We define $\cCu(X) \subset \cR(X)$ to be the subring generated by the following tuples $f = (f_K)_K$; as usual, $k_K$ is the residue field
of $K$ and $q_K$ is the cardinality of $k_K$.
\begin{enumerate}
 \item
   $f_K(x) = \#Z_{K,x}$, where $Z \subset X \times \RF^r$ is a definable set (for any $r$)
 \item 
   $f_K(x) = 1/h(q_K)$, where $h \in \ZZ[q]$ is a polynomial
 \item
   $f_K(x) = \alpha_K(x)$, where $\alpha\colon X \to \VG$ is a definable function
 \item
   $f_K(x) = q_K^{\alpha_K(x)}$, where $\alpha\colon X \to \VG$ is a definable function
\end{enumerate}
\end{defn}

\begin{rem}
It might seem strange that we have both, (3) and (4). However, this is necessary to make the rings $\cCu(X)$
closed under integration. Intuitively, think of (3) as the logarithm of (4) and recall that in the reals,
integrating $1/x$ yields $\log x$. See also Example~\ref{ex.log}.
\end{rem}

Now let us already verify that we can use this to measure definable sets:

\begin{proof}[Proof that Theorem~\ref{thm.intu} implies Theorem~\ref{thm.measu}]
Given a definable set $X \subset \VF^n$, we apply Theorem~\ref{thm.intu} to the constant $1$ function on $X$
(which lies in $\cCu(X)$ by any of (1)--(4)). We obtain $g \in \cCu(\Pt)$ (where $\Pt$ is the one-point definable set),
and this $g$ (which is just a tuple consisting of one real number for each $K$) is just equal to $\muu(X)$; thus it remains to verify that $\cCu(\Pt) = \cCu^0$.

It is clear that $\cCu^0$
is equal to the ring generated by (1) and (2) of Definition~\ref{defn.cCun} (for $X = \Pt$). That (3) and (4)
do not yield anything new in $\cCu(\Pt)$ follows from quantifier elimination (Theorem~\ref{thm.qe}).
(The key step here is that any definable $\alpha\colon \Pt \to \VG$ takes only finitely many values for varying $K$.)
\end{proof}

Note that it suffices to prove Theorem~\ref{thm.intu} in the case $n = 1$: To obtain
the result for bigger $n$, we can then simply integrate out one variable after the other.
Thus, by formulating Theorem~\ref{thm.intu}, we indeed managed to reduce the proof of 
Theorem~\ref{thm.measu} to a problem which is essentially one-dimensional.

\begin{rem}
It might have been tempting to define $\cCu(X)$ differently, namely as the ring of functions in an expansion of the
valued field language having $\RR$ as a new sort. However, we do need $\cCu(X)$ to contain
all the generators listed in Definition~\ref{defn.cCun}, and 
the generators (3) and (4) would then allow to define new, strange subsets of the valued field.
\end{rem}

\begin{rem}
The rings $\cCu(X)$ as defined above are the smallest (non-trivial) ones which are closed
under integration (i.e., which satisfy Theorem~\ref{thm.intu}).
If one would like to
integrate other functions uniformly in $K$, one can also choose bigger rings.
In particular, there exists a version of uniform $p$-adic integration
where the rings contain additive characters $K \to \CC$; this version has
various applications to representation theory.
\end{rem}

\subsection{Deducing rationality of Poincaré series}
\label{sect.proof.rat}

Recall that one of our goals was to prove Theorem~\ref{thm.murat.u} about the rationality of
series obtained from the measure of a family of definable sets parametrized by $\VG_{\ge0}$.
We will now see that this follows from Theorem~\ref{thm.intu}. Given
$X \subset \VG_{\ge 0} \times \VF^n$, by applying Theorem~\ref{thm.intu}, we obtain
that the measures $g_K(s) := \mu_K(X_{K,s})$ form a motivic function $g \in \cCu(\VG_{\ge 0})$.
Thus Theorem~\ref{thm.murat.u} is implied by the following result:

\begin{thm}\label{thm.urat}
For every $f \in  \cCu(\VG_{\ge 0})$, there exist 
definable sets $Z_0, \dots, Z_d \subset \RF^r$ 
and a polynomial $h \in \QQ[q,T]$ such that
\begin{equation}\label{eq.urat}
\sum_{s = 0}^\infty f_K(s) T^s = \frac{\sum_{i=0}^d \#Z_{i,K}T^i}{h(q_K,T)},
\end{equation}
for all $K$ with $p \gg 1$.
\end{thm}

Now note that all definable ingredients to $f$ only live in the value group and the residue field. From quantifier elimination, one can deduce that
there is essentially no definable connection between the residue field and the value group ($\RF$ and $\VG$ are ``orthogonal'').
This allows us to reduce the proof of Theorem~\ref{thm.urat} to a pure computation in the value group: We can assume that
$f$ is a product of generators of type (3) and (4) from Definition~\ref{defn.cCun} and that the functions
$\alpha$ appearing there are definable purely in $\VG$ (and hence do not depend on $K$).
To prove the rationality of a series obtained in this way, one uses that the language on $\VG$ is Presburger arithmetic,
which is well understood. In particular, using that definable functions $\NN \to \ZZ$ are eventually linear on congruence classes,
one reduces to series of the form
\[
\sum_{\substack{s \in \NN\\s \equiv \lambda \mod m}} s^a q_K^b T^s
\]
for $\lambda, m, a \in \NN$, $b \in \ZZ$, and a standard computation shows that such a series is a rational function in $q_K$ and $T$.

\section{Closedness under integration}
\label{sect.proof}

We reduced all our goals to proving Theorem~\ref{thm.intu}, and by the remark at the end of
Subsection~\ref{sect.intu}, it suffices to be able to integrate out a single variable:
Given a motivic function $f \in \cCu(X)$ for $X\subset S \times \VF$, we need to show that the function obtained by integrating
out the $\VF$-variable lies in $\cCu(S)$.
In this section, we will see the main ideas of how this works.
I will start by explaining the case where $f$ is the constant $1$ function on $X$; in other words, we prove:

\begin{prop}\label{prop.meas.cCu}
Suppose that $S$ and $X \subset S \times \VF$ are definable and consider $g = (g_K)_K$
given by
\[
g_K(s) = \mu_K(X_{K,s}).
\]
Then $g \in \cCu(S)$.
\end{prop}

The main ingredient to the proof of this is cell decomposition: Any definable subset $X \subset \VF$
can be written as a finite disjoint union of certain kinds of simple sets called ``cells''.
The measure of a cell is easy to compute explicitly.
This also works in families, and then it yields Proposition~\ref{prop.meas.cCu}

The strategy to treat arbitrary functions $f \in \cCu(S \times \VF)$ is similar, using a refinement
of the Cell Decomposition Theorem which allows us to partition $\VF$ into cells in such a way that also a given function $f \in \cCu(\VF)$
is simple on each cell,
in particular allowing us to compute the integrals explicitly. Again this also works in families and it yields Theorem~\ref{thm.intu}.


\subsection{Measuring using cell decomposition}
\label{sect.proof.meas}

There are various versions of the Cell Decomposition Theorem in valued fields. For simplicity, I start stating a non-family version for a fixed local field $K$.

\begin{thm}\label{thm.cell.K}
For every definable set $X \subset \VF$ and for every $K$ with $p \gg 1$ (the bound depending on $X$), $X_K$ can be written as a finite disjoint union of cells.
\end{thm}

\begin{defn}\label{defn.cell}
A \emph{cell} $C \subset K$ is either
\begin{enumerate}
 \item a singleton $C = \{c\}$, or
 \item a set of the following form:
\[
\{c + x \mid \alpha < v(x) < \beta, v(x) \equiv \lambda \mod m, \ac(x) \in Z\}
\]
for some $c \in K$, $\alpha \in \ZZ \cup \{-\infty\}$, $\beta \in \ZZ \cup \{+\infty\}$, $m \in \NN_{\ge 1}$, $\lambda \in\{0, \dots, m-1\}$, $Z \subset k^\times$.
\end{enumerate}
\end{defn}

\begin{ex}
The set of squares in the valuation ring is a typical example of a cell; see Example~\ref{ex.qu} and Figure~\ref{fig.squares}.
\end{ex}

The measure of such a cell $C$ is easy to compute (in the same way as we computed the measure of the set of squares in Example~\ref{ex.qu}): If $C$ is a singleton, then $\mu_K(C) = 0$, so we only need to deal with the case (2).
The ``$c\,\,+$'' does not change the measure, so we can ignore it.
Then there are two conditions on $v(x)$. Let us first fix an $r \in \ZZ$ satisfying those conditions and look at
the corresponding set
\begin{equation}\label{eq.cellball}
\{x \in K \mid v(x) = r, \ac(x) \in Z\}.
\end{equation}
This is a disjoint union of $\#Z$ many balls, each of which has measure $q^{-r-1}$. Thus the total measure of $C$ is
\begin{equation}\label{eq.mucell}
\mu_K(C) = \#Z\cdot \sum_{\substack{\alpha < r < \beta\\r \equiv \lambda \mod m}} q^{-r-1}.
 \end{equation}
If $\alpha = -\infty$, then that sum is infinite. Otherwise, let us first assume that $\beta = +\infty$.
Then the sum can be rewritten as
\begin{equation}\label{eq.musum}
\mu_K(C) = \#Z \cdot \sum_{j = 0}^\infty q^{a-mj} = \frac{\#Z \cdot q^a}{1 - q^{-m}}
 \end{equation}
for some suitable $a \in \ZZ$. Finally, if $\beta$ is also finite, then (\ref{eq.mucell}) is equal to the difference of two expressions of the form (\ref{eq.musum}).

Now if we do all this in families and for varying $K$, we would like to say that the various ingredients to the definition of a cell --
namely $c$, $\alpha$, $\beta$, $m$, $\lambda$, $Z$ -- are definable. Actually, one can even assume that $m$ and $\lambda$ are constant
(using a compactness argument and by encoding a partition of $S$ into $Z$).
So we could hope for the following result:

\begin{almthm}\label{thm.cell.u}
Suppose that $S$ and $X \subset S \times \VF$ are definable sets. Then $X$ can be partitioned into
finitely many ``cells over $S$''.
\end{almthm}

\begin{defn}\label{defn.cell.u}
Fix a definable set $S$. A \emph{cell over $S$} is a definable set $C \subset S \times \VF$ of one of the following two forms:
\begin{enumerate}
 \item $C = \{(s, c(s)) \mid s \in S'\}$ for some definable set $S' \subset S$ and some definable function $c\colon S' \to \VF$.
 \item $C = \{(s, c(s) + x) \mid s \in S, \alpha(s) < v(x) < \beta(s), v(x) \equiv \lambda \mod m, \ac(x) \in Z_s\}$
   for some definable set $Z \subset S \times \RF^\times$, some definable functions
     $c\colon S \to \VF$, $\alpha,\beta\colon S \to \VG \cup \{\pm\infty\}$ and some integers $\lambda$ and $m$.
\end{enumerate}
\end{defn}

%

Unfortunately, Almost-Theorem~\ref{thm.cell.u} is only almost true. For example, consider $S = \VF$ and $X = \{(s, x) \in \VF^2 \mid x^2 = s\}$.
Then whenever $s \in S_K$ is a non-zero square, the fiber $X_{K,s}$ consists of two points (and hence is a union of two cells in the sense
of Definition~\ref{defn.cell}), but there is no definable way of separating this into two cells over $S$.
However, in some sense, this is the only aspect of Almost-Theorem~\ref{thm.cell.u} which is false, and for our purposes,
this is harmless, since whenever
several cells cannot be separated, they all have the same measure.
For this reason, in these notes, I will cheat and simply use the above almost-theorem.

Since I claimed (in Subsection~\ref{sect.qe}) that quantifier elimination is the only ingredient we use,
let me mention that it is not too difficult to deduce (the correct version of) Almost-Theorem~\ref{thm.cell.u} from Theorem~\ref{thm.qe} (though in the original
article \cite{Pas.cell} by Pas, it is done the other way round: quantifier elimination is deduced from cell decomposition).

\begin{proof}[Proof of Proposition~\ref{prop.meas.cCu}]
Since $X$ is a finite disjoint union of cells over $S$, and using the computation below Definition~\ref{defn.cell},
we obtain that $\mu_K(X_{K,s})$ is a sum of expressions of the form
\begin{equation}\label{eq.musum.fam}
\pm \frac{\#Z_s \cdot q_K^{a_K(s)}}{1 - q_K^{-m}}
\end{equation}
for some definable $Z \subset S \times \RF^\times$ and $a\colon S \to \VF$.
(The presence of some of the summands might depend on whether $\beta = +\infty$ or not,
but to make a summand disappear for some of the $s$, one can simply choose the corresponding $Z_s$ to be empty.)

Now (\ref{eq.musum.fam}) is indeed a product of factors of the form (1), (2) and (4) of Definition~\ref{defn.cCun}.
\end{proof}

\subsection{Integrating using cell decomposition}
\label{sect.closed.gen}

Now suppose we have a motivic function $f \in \cCu(X)$ for $X \subset S \times \VF$ and want to prove
that integrating out the $\VF$-variable yields a motivic function in $\cCu(S)$.
For this, we use a version of the Cell Decomposition Theorem which provides
cells that are ``adapted to $f$''. More precisely, recall from Definition~\ref{defn.cCun} that there are two kinds
of ingredients making functions in $\cCu(X)$ non-constant:
definable sets $Z \subset X \times \RF^r$ appearing in (1), and definable functions 
$\alpha\colon X \to \VG$ appearing in (3) and (4).

A cell decomposition can be adapted to such objects in the following sense. Again, for simplicity,
I state a non-parametrized single-field version:

\begin{thm}
Suppose that we are given a definable set $X \subset \VF$, finitely many definable sets $Z_i \subset X \times \RF^{r_i}$
and finitely many definable functions $\alpha_j\colon X \to \VG$. Then for every $K$ with $p \gg 1$,
$X_K$ can be written as a finite disjoint union of cells such that moreover, for each cell of the form
\begin{equation}\label{eq.cell.form}
C = \{c + x \mid \alpha < v(x) < \beta, v(x) \equiv \lambda \mod m, \ac(x) \in Z\}
\end{equation}
(as in Definition~\ref{defn.cell}), we have:
\begin{enumerate}
 \item for each $i$, the fiber $Z_{i,K,c+x}$ only depends on $\ac(x)$ (for $c+x \in C$);
 \item for each $j$, the function value $\alpha_{j,K}(c+x)$ only depends on $v(x)$ (for $c+x \in C$).
\end{enumerate}
\end{thm}

\begin{rem}
Actually, this formulation is slightly imprecise, since it might be possible to write a cell in the form
(\ref{eq.cell.form}) in different ways. One really should say: Each cell $C$ can be written in the form (\ref{eq.cell.form})
in such a way that (1) and (2) hold. 
\end{rem}

Now given $X \subset \VF$ and $f \in \cCu(X)$, a similar computation as the one below Definition~\ref{defn.cell} can be used to determine
the integral $\int_C f_K(x)\,dx$ over a cell $C \subset X_K$ adapted to all the ingredients of $f$: First, we
neglect the ``$c\,\,+$'' of the cell and we
write the integral as a sum of separate integrals over the sets $\{x \in K \mid v(x) = r, \ac(x) \in Z\}$.
The residue field ingredients to those integrals can be pulled out of the entire sum, so that we are left with
an expression involving only the ingredients $\alpha_{j,K}$ of $f$, which now can be considered as functions in $r = v(x)$.
In particular, the functions $r \mapsto \alpha_{j,K}(x)$ are Presburger definable, and 
we can finish using the same technique as in the proof of rationality of series in Subsection~\ref{sect.proof.rat}.
Instead of giving more details, let me give two examples:

\begin{ex}\label{ex.bla}
Suppose that $C = \{x \mid 0 \le v(x), \ac(x) = 1\}$ and $f_K(x) = v(x)$.
Then for $0 \le i < \beta$, the integral of $f_K$ over the ball $B_r := \{x \mid v(x) = r, \ac(x) = 1\} \subset C$
is equal to $\mu(B_r) \cdot r = q^{-r-1} \cdot r$.
Thus
\[
\int_{C_K} f_K(x)\,dx = \sum_{r = 0}^{\infty} q^{-r-1} \cdot r = \frac{1}{(q-1)^2}.
\]
\end{ex}

\begin{ex}\label{ex.log}
Suppose that $C = \{x \mid 0 \le v(x) < \beta, \ac(x) = 1\}$ and $f_K(x) = q^{v(x)}$.
Then for $0 \le r < \beta$, the integral of $f_K$ over the ball $B_r := \{x \mid v(x) = r, \ac(x) = 1\} \subset C$
is equal to $\mu(B_r) \cdot q^{r} = q^{-r-1} \cdot q^{r} = q^{-1}$. Thus $\int_{C} f_K(x)\,dx = \beta\cdot q^{-1}$.
Since when we will look at this in families, $\beta$ will be a definable function of the parameters,
one sees how functions of the form (3) in Definition~\ref{defn.cCun} arise.
\end{ex}

As in Subsection~\ref{sect.proof.meas}, all this also works in families, thus
yielding a proof of Theorem~\ref{thm.intu}.

\section{Motivic integration in other valued fields}
\label{sect.mot}

To end these notes, I will explain how one obtains a version of motivic integration which works in other valued fields than
local ones. There are several different approaches to this. There is a very nice survey by Hales \cite{Hal.motMeas} about
the original approach by Kontsevich. Below, I sketch two more modern approaches by
Cluckers--Loeser \cite{CL.mot} and by Hrushovski--Kazhdan \cite{HK.motInt}.

\subsection{Cluckers--Loeser motivic integration}

This version of motivic integration is designed for valued fields of the form $K = k((t))$, for arbitrary $k$ of characteristic $0$.
The idea is to define, for every definable set $X$, a ring $\cCmot(X)$ which is an abstract analogue of our $\cCu(X)$:
instead of being a ring of tuples of functions, $\cCmot(X)$ is given in terms of generators and relations.
Moreover, now, when we speak of definable sets (e.g.\ concerning the above $X$), instead of
working in the theory $\Tloc$, we work in the theory I denoted by $T_0$ in Subsection~\ref{sect.th}: the theory of
henselian valued fields with value group elementarily equivalent to $\ZZ$ and with residue characteristic $0$.

Specifying the generators of $\cCmot(X)$ in analogy to Definition~\ref{defn.cCun} is easy: 
For example, as an analogue of Definition~\ref{defn.cCun} (1), we have one generator for every definable set $Z \subset X \times \RF^r$;
a natural notation for this generator is ``$x \mapsto \#Z_x$'', even though this has no real meaning now.
Similarly, we have generators (2) ``$x \mapsto 1/h(\#\RF)$'', (3) ``$x \mapsto \alpha(x)$'' and (4) ``$x \mapsto \#\RF^{\alpha(x)}$''
for polynomials $h$ and definable maps $\alpha\colon X \to \VG$.

A more subtle task consists in finding the right relations for $\cCmot(X)$. I will not list all of them here,
but let me just say that all of them are natural if one thinks of the intended meaning. For example, 
the (1)-generator
``$x \mapsto \#\RF$'' is equal to the (4)-generator ``$x \mapsto \#\RF^{1}$''.

\begin{rem}
Deciding which relations to use exactly is not entirely straightforward.
For example, if $Z$ is the set of non-zero squares in the residue field and $Z'$ is the set of non-squares,
then $\#Z_K = \#Z'_K$ for all $K$ with residue characteristic $\ge 3$, so that $Z$ and $Z'$ yield the same
element of $\cCu(\Pt)$. Nevertheless, they should not be made equal in $\cCmot(\Pt)$, intuitively because
in $K = \CC((t))$, $Z = \CC^\times$ but $Z' = \emptyset$.
\end{rem}

In the uniform $p$-adic setting, we proved that the rings $\cCu(X)$ are closed under integrating out
some of the variables. For the rings $\cCmot(X)$, it is not even clear what integration
is supposed to be. What one does is: one defines motivic integration maps between the different rings $\cCmot(X)$
which mimic the computations we did for $p$-adic integration. For example, one defines an
integration map $\cCmot(\VF) \to \cCmot(\Pt), f \mapsto$ ``$\int_{\VF} f$'' as follows.
\begin{itemize}
 \item Choose a cell decomposition of $\VF$ adapted to $f$ and integrate $f$ on each cell separately.
   (The integral $\int_{\VF} f$ is then defined to be the sum of the integrals over the cells.)
 \item The integral of $f$ over a cell adapted to $f$ is defined explicitly, in analogy to the computations
    sketched at the end of Subsection~\ref{sect.closed.gen}.
\end{itemize}

\begin{ex}
Suppose that $X = \{x \in \VF \mid 0 \le v(x), \ac(x) = 1\}$ and $f(x) = v(x)$. By analogy to Example~\ref{ex.bla}, one
defines $\int_X f := \frac1{(\#\RF - 1)^2}$.
\end{ex}

For this definition to make sense, one has to verify that it does not depend on the chosen cell decomposition.
Moreover, one would like to know that motivic integration does indeed behave like integration:
It should satisfy the Fubini theorem (i.e., when integrating out several variables, the order of the variables should not matter)
and a change of variables formula. All these things were trivial in the uniform $p$-adic case, since
there, integration was just field-wise Lebesgue integration (for which all of this holds). In the
motivic setting, proving these things is the main work. (Note that for these things to hold, it is important that the
rings $\cCmot(X)$ were defined using the right relations.)

For various applications to algebraic geometry (like the one by Kontsevich mentioned in the introduction), it is enough to have any theory of (motivic)
integration which has the above properties.
Moreover, this kind of motivic integration can replace the uniform $p$-adic integration introduced earlier
in these notes, since we have natural maps $\cCmot(X) \to \cCu(X)$ commuting with integration.
(This should be clear from the way we defined motivic integration.)
Nevertheless, it would be more satisfactory if we knew that our notion of motivic integration is also
in some sense natural and/or unique. Cluckers--Loeser prove that in some sense it is, but
the approach by Hrushovski--Kazhdan provides a much nicer result of this kind, so now I will explain their approach.

\subsection{Hrushovski--Kazhdan motivic integration}

Hrushovski and Kazhdan introduced two new ideas to the theory of motivic integration. One is that
one can simplify things by working in algebraically closed valued fields instead of henselian ones.
(One can then nevertheless deduce results about non-algebraically closed fields.) The other one is
to define motivic integration by a universal property making it
``the most general theory of integration in valued fields''. In these notes, I will only consider the second idea:
I will stick to valued fields of the form $K = k((t))$ but explain how the universal property approach works.

For simplicity, let us go back to the point of view that to integrate, one just needs a measure.
Thus we simply want to define ``the most general map from the class of definable sets into a ring $\cCmot^0$ which behaves like a measure''.
Formally, this means that we let $\cCmot^0$ be generated by elements $[X]$ for all definable sets $X$, and we quotient by
the relations a measure is supposed to satisfy. For example, if $X, Y \subset \VF^n$ are disjoint, then $[X \cup Y] = [X] + [Y]$,
and if we have a ``measure-preserving'' bijection $\alpha\colon X\to Y$, then $[X] = [Y]$.
(One needs to define which bijections should be considered as measure-preserving. This is done in analogy to
the $p$-adic world; for example, if $X, Y \subset \VF$, then a differentiable map $X \to Y$ whose
derivative has valuation $0$ everywhere is measure-preserving.)

I will not go into the details of how one then defines $\cCmot(X)$ and integration using this approach,
but note that one gets many results for free: one has a well-defined notion of motivic integration, it satisfies
all the properties one would like it to satisfy (Fubini, change of variables), and it specializes to
Cluckers--Loeser motivic integration, simply because Cluckers--Loeser motivic integration satisfies all the properties used by Hrushovski--Kazhdan
in the definition of $\cCmot^0$ (and of the $\cCmot(X)$).

This time, however, the challenge is to determine $\cCmot^0$, and more generally $\cCmot(X)$;
otherwise, the definition of motivic integration is just useless general non-sense.
In the setting Hrushovski and Kazhdan work in, namely for algebraically closed valued fields,
the ring $\cCmot^0$ is a bit more complicated than the one of Cluckers--Loeser:
Whereas the Cluckers--Loeser-$\cCmot^0$ is a kind of Grothendieck ring of definable sets in the residue field,
the Hrushovski--Kazhdan-$\cCmot^0$ also uses definable sets in the value group.
However, by work in progress, it seems that if one applies the universal construction of
Hrushovski--Kazhdan to the fields $K = k((t))$, then one 
obtains exactly the same rings $\cCmot(X)$ as with the definition of Cluckers--Loeser.
In other words, after all, Cluckers--Loeser motivic integration was already natural and as general as possible.

\bibliographystyle{siam}
\bibliography{references}

\end{document}